\documentclass[a4paper,12pt]{amsart}
\usepackage{amsmath}  
\usepackage{amssymb}
\usepackage{amsthm}  
\usepackage{amscd} 
\usepackage{color}
 
\theoremstyle{plain}
\newtheorem{thm}{Theorem}[section]

\newtheorem{lemma}[thm]{Lemma}
\newtheorem{cor}[thm]{Corollary}

\theoremstyle{definition}
\newtheorem{defn}[thm]{Definition}

\theoremstyle{remark}
\newtheorem{rem}{Remark}

\usepackage{tikz-cd}

 \usepackage{dynkin-diagrams}

\title[Transforms of holomorphic maps]
{Transforms of holomorphic maps from flag manifolds into Grassmannians}
\author{Oscar Macia, Yasuyuki Nagatomo}
\address{Department of Mathematics, University of Valencia,
  C. Dr. Moliner S/N, Burjassot (46100), Valencia, SPAIN}
\email{oscar.macia@uv.es}
\address{Department of Mathematics, Meiji University, 
  Higashi-Mita, Tama-ku, Kawasaki-shi, Kanagawa 214-8571, JAPAN}
\email{yasunaga@meiji.ac.jp}
\date{}

\begin{document}

\maketitle

\begin{abstract}
 Building on the generalized do Carmo–Wallach theory, we construct a functorial framework for holomorphic maps from flag manifolds into Grassmannians and quadrics.
Through the direct image sheaf and its inverse, we construct Penrose-type transforms that allow both the domain and target to vary. These transforms provide  
functors between categories of full holomorphic maps satisfying the gauge condition for semi-positive homogeneous bundles. 
Within this framework, Einstein–Hermitian holomorphic maps form a natural subcategory, stable under all transforms, and the minimality of $L^2$-norm of the mean
curvature operator is preserved. 

For Grassmannian targets, the moduli of Einstein-Hermitian maps from a flag manifold are identified with $r$‑tuples of non-positive integers modulo symmetry, 
where $r$ is the rank of the group of holomorphic isometries. 
For quadric targets, we obtain a complete geometric description of the moduli space. 
The center of this moduli space corresponds to the Einstein-Hermitian map into the projective space, 
and the tower of transforms identifies all intermediate moduli spaces with the moduli of holomorphic isometric embeddings at the bottom level.  
The results yield a unified categorical and geometric description of holomorphic isometric embeddings of flag manifolds.
\end{abstract}

\section{Introduction}

In 1953, Calabi introduced the diastasis function in his study of holomorphic isometric immersions
of Kähler manifolds into complex projective space $\mathbf C \mathrm P^n,$  proving the rigidity
of such immersions \cite{Cal}. Using Calabi’s rigidity theorem, Takeuchi later classified holomorphic
isometric embeddings of flag manifolds into $\mathbf C\mathrm P^n$ \cite{takeuchi}.
When the second Betti number of the flag manifold is $s,$ the corresponding moduli space can
be identified with $\mathbf{Z}^{s}_{+}$ modulo a certain symmetry, where $\mathbf{Z}^{s}_{+}$ 
denotes the set of  $s$-tuples of positive integers.
On the other hand, using Calabi’s diastasis function, Suyama showed that the moduli space of
holomorphic isometric immersions into a quadric hypersurface of $\mathbf C\mathrm P^n$
 is given by the quotient of a certain subset of complex Euclidean space by a circle action \cite{Suy}.

A theory generalising that of do Carmo and Wallach describing harmonic maps in terms of vector bundles has been developed in
\cite{Na13}. 
This theory has been successfully applied to describe the moduli space of various classes of special maps from  families of projective varieties into 
 complex Grassmannians and quadrics. 
Peng-Xu classify $SU(2)$-equivariant minimal isometric immersions of $\mathbf{C}\mathrm{P}^1$ 
into the complex Grassmannian of $2$-planes in an algebraic way \cite{Pen-Xu}. 
Then Koga and the second author apply a generalisation of do Carmo-Wallach to 
obtain those classifications \cite{Kog-NagII} and \cite{Kog-NagIII} 
and this result is generalised to the case where the domain manifold is 
$\mathbf{C}\mathrm{P}^n$ \cite{Koga_Takahashi, Na-18}. 
Furthermore, it enables us to generalise Calabi's rigidity and leads to the result 
that the category of unitary representations of $G$ is isomorphic to the category of full 
holomorphic maps from a flag manifold  to complex Grassmannians satisfying the gauge condition
for semi-positive homogeneous vector bundles, 
which is considered as a generalisation of Takeuchi's result \cite{NaCat}.  

 As for holomorphic isometric embeddings to quadrics, 
a careful application of representation theory
 and an improved understanding of the underlying principles of the generalised do Carmo-Wallach theory, permitted successive generalisations of the domain spaces.
Thus, the case of $\mathbf C \mathrm P^1$  was firstly studied in
\cite{MNT}. Later, innovations in \cite{Na21} allowed the generalisation to $\mathbf C \mathrm P^n$ and opened the possibility of studying more general projective varieties of Grassmannian type. 
These new ideas were put to the test first in the case of $Gr_2(\mathbf C^4),$ and later extended to arbitrary Grassmannians in \cite{MN22, MNxx}.

In the present paper, 
from the aforesaid isomorphism of categories, 
we obtain a transform of holomorphic maps as functors of the categories.   
More precisely, we provide a Penrose-type transform using the direct image sheaves and the pull-back
of the vector bundles.
Different from the well-known transforms of harmonic maps 
(eg. harmonic sequences, B\"acklund transforms),  
both the domain and the target manifolds of holomorphic maps could vary 
under the transform.  

A subcategory of Einstein-Hermitian maps (EH maps for short) is preserved under the transforms. 
The EH map is characterised within its homotopy class by the fact that its associated mean
curvature operator attains the minimal $L^2$ norm.

When the target is a projective space, the functors provide  a tower of transforms of EH maps.
The domain at the top of the tower is the full flag manifold.
The bottom map is an isometric embedding;
its domain is determined by the degree of the pull-back of the hyperplane bundle over the projective space 
by the initial EH holomorphic map from the full flag manifold. 

When the authors constructed the moduli space of holomorphic isometric embeddings of 
Grassmannians into quadrics \cite{MNxx},  
in the last generalising step, we needed quite detailed computations in representation theory to identify the representation space in which
the moduli space is realised. 
Still, one could think that similar, if somewhat more complicated computations, could be used to further expand these results in order to tackle the question of holomorphic isometric embeddings of flag manifolds 
into quadrics. 
The present article also reports on these investigations; interestingly enough, results are obtained through a general structure theory of these mappings  without resorting to
the heavy representation-theoretic computations which were necessary in the general Grassmannian case. 
Our methods provide an explicit description of the moduli space in this case, and
the center of the moduli space is identified with the EH map into the projective space.

 Different from the moduli spaces of those maps into the  complex Grassmannian, 
 non-discrete sets emerge as moduli spaces. 
 Nonetheless, the generalisation of do Carmo-Wallach theory enables us to construct the tower of transforms even in this case, and to  identify
 the moduli spaces at each level of the tower
 with the moduli space of holomorphic isometric embeddings at the bottom of the tower. 
The bottom manifold is again determined by the degree of the pull-back of the 
universal quotient bundle over the quadric by the initial EH holomorphic map.
The tower of transforms preserves the minimality of the $L^2$ norm associated to the 
mean curvature operators.

The paper is organised as follows: Section 2  is devoted to a review of the 
basic notions of Einsten-Hermitian holomorphic maps  
and 
certain basic results of the generalisation of the theory of do Carmo-Wallach
necessary to the description of moduli spaces in this setting.  
 In \S3, we fix our notations for flag manifolds and exhibit the 
moduli spaces of Einstein-Hermitian holomorphic maps of 
flags into complex Grassmannians: those are identified with 
$r$-tuples of non-positive integers modulo some symmetry, 
where $r$ is the rank of the 
groups of  holomorphic isometries on flags (compare to Takeuchi's result). 
The transforms of maps are developed and applied to 
the identification of moduli spaces in \S4. 
In \S5, the transforms are restricted to the case 
where the target spaces are the projective spaces 
and, as a consequence,  each moduli space is in one-to-one correspondence with 
the moduli space of holomorphic isometric embeddings of the bottom flag. 
In \S6, we use a generalisation of do Carmo-Wallach theory to 
describe the moduli spaces of 
Einstein-Hermitian holomorphic maps of 
flags into quadrics.   In \S 7 we construct the tower of transforms when the target
space is a quadric.\\

{\bf Acknowledgements.} The research of Y.N. is supported by JSPS KAKENHI Grant Number 26K06816. O.M. would like to thank the hospitality of Meiji University (Japan), where part of this research was developed.

%%%%%%%%%%%%%%%%%%%%%%%%%%%%%%%%%%
%
%%%%%%%%%%%%%%%%%%%%%%%%%%%%%%%%%%
\section{Preliminaries}
We will use the present section to introduce the basics of Grassmannians, evaluation homomorphisms and induced maps. The discussion is similar to that in \cite{MNT} \S 2, but this time restricted to the case of holomorphic vector bundles.
\subsection{Geometry of Grassmannians}
Let $W$ be a complex  
vector
space and $Gr_p(W)$ the complex  Grassmannian of
$p$-planes in $W.$ Generically, $\underline W$ will stand for the total space of a trivial vector
bundle $\underline W \to B$ with fibre $W$ over some specified base manifold $B.$ Denote by
$\underline W \to Gr_p(W)$ the trivial bundle of fibre $W$ over $Gr_p(W).$ Then, there is a natural
exact sequence 
\[
0 \to S \to  \underline W \to Q \to 0
\]
of vector bundles over $Gr_p(W)$
where $S$ is the total space of the tautological bundle, and $Q$ is the total space of the universal quotient bundle. 
Notice that by using the natural projection 
\[
\pi_Q:\underline W\to Q, 
\]
we can regard $W$ as the space $H^0(Gr_p(W);Q)$ of holomorphic sections of $Q\to Gr_p(W).$

By fixing a Hermitian inner product on $W,$ both the tautological subbundle $S$ and the universal quotient bundle $Q$
inherit induced Hermitian metrics and canonical Hermitian connections.
The second fundamental forms 
\[
H\in\Omega^{1,0}(\mathrm{Hom}(S,Q))\quad  
\text{and}\quad  K\in\Omega^{0,1}(\mathrm{Hom}(Q,S))
\]
in the sense of Kobayashi \cite{Kob} are also obtained. 
The curvature $R^Q$
of the canonical Hermitian connection on $Q$ can then be expressed as 
\[R^Q=-H  K\]
from the Gauss-Codazzi equations on vector bundles, and  the K\"ahler form on the Grassmanian is given
as
\[
\omega_{GR}=-\sqrt{-1}\; \mathrm{tr}\; R^Q.
\]

\subsection{Evaluation homomorphism and induced maps}
Let $M$ be a complex or K\"ahler manifold.
Suppose that $V \to M$ is a holomorphic vector bundle 
admitting non-trivial holomorphic sections, and consider  
the space of holomorphic sections $W = H^0(M;V).$ By definition
of $\underline W \to M,$ there is a bundle homomorphism 
\[ 
ev : \underline W \to V,
\] 
called the evaluation map,
given by $(x, t)\mapsto  t(x)$ for all $t \in W,$ $x \in M.$ \\
The vector bundle $V \to M$ is
said to be {\it globally generated by $W$} if the evaluation map is surjective. 
Under this hypothesis, there is a map 
\[
f : M \to Gr_p(W), 
\]
with $p=\dim W - \text{rank}\; V,$ 
defined by 
\[
f(x) := \ker  ev(x,\cdot) = \{t \in W : t(x) = 0\} .
\]
The map $f$ is  called
the {\it induced map by} $(V\to M,W).$\\

\subsection{Maps satisfying the gauge condition}
Let 
\[
f:M\to Gr_p(W)
\] 
be a holomorphic map. 
Then, the natural exact sequence over $Gr_p(W)$ can be pulled back to an exact sequence over $M.$
Of particular interest to us will be the pull-back bundle of the universal quotient bundle, that is, 
\[
f^*Q\to M.
\] 
If the universal quotient bundle $Q\to Gr_p(W)$ is equipped with
a Hermitian fibre-metric $h^Q$ and a connection $\nabla^Q,$ these are also pulled back to  a fibre-metric $f^*h^Q$ and a connection $\nabla^{f^*Q}$ on $f^*Q\to M,$
which is a holomorphic vector bundle over $M.$\\
 The mapping $f:M\to Gr_p(W)$ is said to be {\it full} if the induced linear map 
\[
W\to H^0 (M;f^*Q)
\] 
is a monomorphism.

 A  {\it Hermitian bundle} $(V,h)$ is a holomorphic vector bundle  over a K\"ahler manifold, $V\to M,$  equipped with  a {\it Hermitian fibre-metric} $h.$
 Each Hermitian bundle admits a unique connection  compatible with the fibre-metric and the holomorphic structure.
 This connection will be called the {\it Hermitian connection} of $(V,h).$ 

Assume that a Hermitian holomorphic bundle 
\[
(V,h)\to M
\] 
is given. 
 We will say
that $f:M\to Gr_p(W)$ {\it satisfies the gauge condition for
$(V,h)$} if there exists an isomorphism 
\[
(V,h)\cong (f^*Q,\;f^*h^Q)
\]
as Hermitian holomorphic bundles.

A Hermitian bundle $(V,h)$ is said to be semi-positive 
if 
\[h(R^\nabla (Z,\bar Z)(v), v)\geq 0,\]
for arbitrary non-zero tangent vector of type $(1,0)$ $Z$ and non-zero $v\in V,$
where $R^\nabla$ is the curvature of the Hermitian connection.

Let $SU(W)$ be the group of special unitary transformations of $W,$ acting as the isometry group of $Gr_p(W).$
Two mappings $f_1,f_2:M\to Gr_p(W)$ are {\it image equivalent} if there is an isometry 
$\phi\in SU(W)$  
of $Gr_p(W)$ such that
\[
f_2=\phi\circ f_1.
\] 

\subsection{Einstein-Hermitian holomorphic maps}

The contraction of the curvature of the Hermitian connection with
the K\"ahler form of the base defines the {\it mean curvature} $K_{EH}$ in the
sense of Kobayashi \cite{Kob}, which is a section of the bundle of endomorphisms of $V.$ 
If the mean curvature 
$K_{EH}$ of $(V, h)$ 
satisfies 
\[
K_{EH} =\lambda I_V
\] 
for some constant $\lambda,$ then
$(V, h)$ is said to be an {\it Einstein-Hermitian vector bundle} 
and the corresponding  Hermitian connection is called a Hermitian
Yang-Mills connection, \cite{Kob,DonKro}.

A holomorphic map 
\[
f:(M, \omega)\to Gr_p(W)
\] 
satisfying the gauge condition for $(V,h)$ is called an
{\it Einstein-Hermitian holomorphic map}  if $(V,h)$ is an 
Einstein-Hermitian bundle \cite{Na13, MacNag, Na15}, where $\omega$ is the K\"ahler form on $M$. 

An Einstein-Hermitian holomorphic map attains the minimum of the $L^2$ norm
$||A||^2,$ in some fixed homotopy class of maps, 
where $A$ is the mean curvature operator of maps \cite{Na13}. 
We do not need the precise definition of $A$ in this article. 
When a map is holomorphic, 
\[
A=-K_{EH}.
\] 
The following result is proven in \cite{Na13}.

\begin{thm}\label{equality_holds}
  If $f:M\to Gr_p(\mathbf C^{p+q})$ is a smooth map of a compact K\"ahler manifold, then
  \[
    \left(\frac{2\pi}{(m-1)!}\int_Mc_1(f^*Q)\wedge \omega^{m-1}\right)^2 \leq q {\rm vol}(M)\int_M|A|^2 dv_M.
  \]
  The equality holds if and only if $f$ is an Einstein-Hermitian holomorphic or anti-holomorphic map
  with EH constant $-|c|,$ where
  \[
c = \frac{2\pi}{q {\rm vol}(M)(m-1)!}\int_M c_1(f^*Q)\wedge \omega_M^{m-1}.
\]
If the target is a quadric hypersurface of the complex projective space, then
\[8\left(\frac{\pi}{(m-1)!}\int_Mc_1(f^*Q)\wedge\omega_M^{m-1}\right)^2\leq {\rm vol}(M)\int_M|A|^2 dv_M.\]
 The equality holds if and only if $f$ is an Einstein-Hermitian holomorphic or anti-holomorphic map
  with EH constant $-|c|,$ where
\[c=\frac{\pi}{{\rm vol}(M)(m-1)!}\int_Mc_1(f^*Q)\wedge \omega_M^{m-1}.\]
\end{thm}

  An Einstein-Hermitian map 
  is isometric
   under special conditions.
  A fundamental result in this direction is the following one:
  \begin{thm}\label{fromagag}
    \cite[Theorems 6.1 and 6.3]{Na15}
    Let $f$ be a holomorphic map from a compact K\"ahler manifold $M$ into the
    complex projective space $\mathbf C \mathrm P^n,$ identified with $Gr_n(\mathbf C^{n+1}),$ or  into the quadric $Gr_n(\mathbf R^{n+2}).$ Suppose that the K\"ahler form $\omega$ on
    $M$ is in the cohomology class represented by the first Chern class of the pull-back of the universal quotient bundle
\[
\mathcal O(1)\to Gr_n(\mathbf C^{n+1}),\quad  \text{or} \quad  
\mathcal O(1)\to Gr_n(\mathbf R^{n+2}). 
\]
Then, the following three conditions are equivalent. \begin{enumerate}
\item The energy density of f is constant.
\item $f$ is an Einstein-Hermitian holomorphic immersion. 
\item $f$ is an isometric immersion.
  \end{enumerate}
\end{thm}

\section{Classification}
\subsection{Flag manifolds and crossed Dynkin diagrams}
In the present paper the relevant compact K\"ahler manifolds of  interest   are flag manifolds.
Let $F$ be a (generalized) flag manifold 
expressed as  $ G/ K$,  where $G$ is a compact, connected, simply connected semisimple Lie group
and 
$K=C(T)$ is the centraliser of a subtorus $T$ of a 
maximal torus $T_0 \subset G.$
Fix a maximal torus $T_0$ in $G$ and a choice of positive roots $\Delta^+$ of $\mathfrak g^c$ (the complexification of the Lie algebra $\mathfrak g$ of $G$).
Then the simple roots of $\mathfrak g^c$
relative to $(T_0,\Delta^+)$  are defined; denote the simple roots by 
\[
\{\alpha_1,\alpha_2,\dots,\alpha_r\}.
\]
 The standard subtorus $T$ (and
hence the flag $F$) can be specified by choosing a subset $\mathcal T$
of the set of  simple roots   of $\mathfrak{g}^c$, and $F$ is denoted by 
$F_G(\mathcal T)$. 
In diagrammatic notation  \cite{BastonEastwood} 
this corresponds to crossing the nodes in the Dynkin diagram
associated to the simple roots in $\mathcal T.$
A sample of such a diagram for a $\mathrm A$-type Lie algebra would be 
\[ \dynkin[scale=1.8]A{XXX}\]
corresponding to the complete flag 
\[
F_{1,2,3}(\mathbf C^4)= F_{SU(4)}(\alpha_1,\alpha_2,\alpha_3),
\] 
defined by choosing $\mathcal T$ as the full set of
simple roots of $\mathfrak{su}(4)$: 
\[
\mathcal T=\{\alpha_1,\alpha_2,\alpha_3\}.
\]

Not all the vector bundles considered in this article are line bundles.
  It is known that all line bundles over generalized flag manifolds are homogeneous bundles; so, as
  a generalization, we consider higher rank homogeneous vector bundles over generalized flag manifolds. 
Let $V_{0}$ be a $K$-module. Then the vector bundle $G\times_{K}V_{0}$ is called a {\it homogeneous vector bundle.}
To describe irreducible homogeneous bundles over the flag manifold, suppose 
that 
\[
\mathcal T=\{\alpha_1,\alpha_2,\dots,\alpha_s\}, s\leq r,
\]
are the simple roots characterising the flag.
We denote by 
\[
\{\varpi_1,\varpi_2,\dots,\varpi_r\}
\] 
the fundamental weights corresponding to 
$\{\alpha_1,\alpha_2,\dots,\alpha_r\},$
\cite{tomDieck}.
If $V_{0}$ is an irreducible $K$-module with 
\[
-\sum_{i=1}^{r} k_i\varpi_i
\] 
as the lowest weight, the vector bundle $G\times_{K} V_{0}$ is called an 
{\it irreducible homogeneous vector bundle with 
$-\sum_{i=1}^{r} k_i\varpi_i$ as the lowest weight} and  
denoted by 
\[
\mathcal O(k_1\varpi_1+k_2\varpi_2+\dots+k_r\varpi_r)\to G/K.
\] 
Especially, if a line bundle $L\to F$ has
\[
-(k_1\varpi_1 +k_2 \varpi_2+\cdots 
+k_s \varpi_{s})
\] 
as its 
weight
then it is denoted by  
\[
\mathcal O(k_1,k_2, \cdots, k_s)\to F. 
\]

By labelling the nodes of the Dynkin diagram with crossed
nodes, vector bundles over flag manifolds 
can also be identified diagrammatically.
To this end we add the label $k_i$ to the node representing the simple root 
$\alpha_i$ in the Dynkin diagram with crossed nodes. Notice that line bundles
\[
\mathcal O(k_1,k_2,\dots,k_s)\to F_G(\mathcal T)
\]
will have labels over crossed nodes only, while general vector bundles
\[
\mathcal O(k_1\varpi_1+k_2\varpi_2+\dots+k_r\varpi_r)\to  F_G(\mathcal T)
\]
will also have labels over uncrossed nodes.

For instance, consider a flag manifold of $\mathrm A_3$-type, and suppose that 
\[
\mathcal T=\{\alpha_1,\alpha_2\}.
\] 
The line bundle 
\[
\mathcal O(1,0)\to F_{SU(4)}(\alpha_1,\alpha_2)
\] 
can be identified with the diagram
\[ \dynkin[labels*={1,,},
scale=1.8]A{XX*}\]
while the vector bundle 
\[
\mathcal O(\varpi_1+\varpi_3)\to F_{SU(4)}(\alpha_1,\alpha_2)
\]
is represented by
\[ \dynkin[labels*={1,,1},
scale=1.8]A{XX*}.
\]

\subsection{A remark on Einstein-Hermitian maps from flags}

  Notice that the definition of Einstein-Hermitian mapping depends on the choice of a K\"ahler form, or metric, on the base manifold.
  The non-uniqueness of this choice poses, however, no problem to the proper definition of Einstein-Hermitian map thanks to the following
  theorem of Kobayashi:
  
  \begin{thm}  
    \cite[Theorems 3.3 and 3.4]{KobHomVBs}
Let $E$ be a holomorphic vector bundle over a compact 
  complex manifold $M$ with an ample line bundle $H$. 
  Let $G$ be a connected compact Lie group of automorphisms of $E$ 
  acting transitively on $M$. 
  Assume that the isotropy subgroup $G_{o}$ of $G$ at a point $o \in M$ 
  acts irreducibly on the fibre $E_{o}$. 
  Then there exists a $G$-invariant Hermitian structure in $E$ and a 
  $G$-invariant K\"ahler metric $g$ on $M$ whose K\"ahler form 
  represents the Chern class $c_1(H)$ of $H$ and 
  $(E,h)\to (M,g)$ is an Einstein-Hermitian vector bundle. 
  
   Assume further that $G$ is semisimple and $G_0$
is the centraliser $C(T)$ of a toral subgroup $T$ of $G.$ Then $(E, h)$ is an
Einstein-Hermitian vector bundle over $(M, g)$ with irreducible holonomy group
and $E$ is $H$-stable for any ample line bundle $H$ over $M.$
\end{thm}
In the case of line bundles, ampleness can be replaced by positivity.
We always suppose that the K\"ahler form is $G$-invariant and represents the first Chern class $c_1(H)$ of
some positive line bundle, so that it is the curvature of the Hermitian Yang-Mills connection.
By virtue of this
theorem, if $f$ is proven to be an Einstein-Hermitian map for this $G$-invariant metric on the base, it will also be Einstein-Hermitian  for any other such
$G$-invariant metric.

  If a vector bundle is not irreducible then there exist various $G$-invariant holomorphic vector bundle structures
  that might be defined on it (for the irreducible case the $G$-invariant holomorphic structure is unique \cite{NaCat}, Lemma 3.1).
  Hence, we fix the holomorphic structure in the following way:
Since every unitary module of  $C(T)$ is completely reducible, any
homogeneous vector bundle $V$ can be decomposed into a direct sum of irreducible
homogeneous vector bundles (as $C^\infty$-bundles) 
\[
V=\oplus_i V_i,
\]
where $V_i$ is an irreducible homogeneous bundle equipped with a Hermitian metric $h_i.$
Thus, the Hermitian metric $h$ on $V$ can be regarded as the direct sum of the metrics
$h_i.$ We further fix  a holomorphic
vector bundle structure on $V$  as a direct sum of holomorphic vector bundles
$(V_i,h_i).$  With this structure, the Hermitian connection on $(V,h)$ is the canonical
connection \cite{KobNo}.
To emphasize the holomorphic vector bundle structure of $V$ as a direct sum of
irreducible homogeneous bundles, we denote the holomorphic homogeneous vector
bundle $(V,h)$ by 
\[\oplus_i(V_i,h_i).
\]
If $i=1,$ that is if $(V,h)=(V_1,h_1),$ then $(V,h)$ is an Einstein-Hermitian vector bundle \cite{KobHomVBs}.

Let $(V_i, h_i)$ ($i=1,2$)  
be homogeneous vector bundles with invariant fiber metrics $h_i$ 
over a flag manifold $G/K$.  
Then $V_1\to G/K$ is said to be {\it isomorphic} to $V_2\to G/K$ 
as a {\it homogeneous bundle}, 
if there exists a $G$-equivariant  
bundle map 
\[
\phi:V_1\to V_2
\] 
preserving the Hermitian holomorphic vector bundle structure.
Such a bundle map $\phi$ is called a {\it $G$-equivariant bundle isomorphism}.

\subsection{Rigidity theorems}

For $G$ a compact, connected, simply connected, semisimple Lie group, fix a toral subgroup $T$ and let $K$ denote its centraliser in $G.$
Denote by ${\mathcal R_G}$ the category whose objects are unitary representations of $G$ and whose morphisms are intertwining isomorphisms between
$G$-representations. Under these hypotheses, for each $W\in \mathcal R_G$ it follows from the Bott-Borel-Weill theorem that
there is a unique semi-positive  homogeneous
vector bundle 
\[
\oplus_i(V_i,h_i)\to G/K 
\]
such that 
\[
H^0(G/K;\oplus_i(V_i,h_i))=W,
\] 
which induces a map \[
f:G/K\to Gr_p(W)
\] 
satisfying the gauge-condition for $\oplus_i(V_i,h_i).$
These maps
form a category denoted by ${\mathcal H_T}$ with $G$-equivariant bundle isomorphisms as
morphisms. The main theorem in \cite{NaCat} establishes that the functor 
\[
\alpha_T:{\mathcal R_G}\to {\mathcal H_T}
\] 
is an isomorphism of categories.

If $(V,h)$ is irreducible, then $f$ is an Einstein-Hermitian holomorphic map.  
The subcategory of $\mathcal{H}_T$ consisting of 
maps 
$f:G/K \to Gr_p(W)$ with the gauge condition for irreducible homogeneous vector bundles
as homogeneous vector bundles will be denoted by $\widetilde{\mathcal{H}}_T.$

\begin{thm}\label{subcat}
  The subcategory 
  $\widetilde{\mathcal{H}}_T$ 
  is isomorphic to the subcategory of
  $\mathcal{R}_G$ consisting of irreducible representations of $G$. 
\end{thm}

\begin{rem}
  As a result, we conclude that the image equivalence classes of Einstein-Hermitian holomorphic maps $f:G/K \to Gr_p(W)$  with gauge condition for irreducible homogeneous vector bundles as homogeneous vector bundles coincide with  the quotient of the set
  $\mathbf Z^r_{\geq 0}$  of non-negative integer $r$-tuples, being $r$ the rank of $G$ 
by some symmetry. 

Notice that $\mathbf Z^r_{\geq 0}$ coincides with the set of unitary irreducible representations of $G$. From this,  the indicated symmetry is easily understood. 
\end{rem}

The set of holomorphic maps 
\[
f:G/K \to Gr_p(W)
\] 
with gauge condition for homogeneous vector bundles $\oplus_i(V_i,h_i)$  
does not depend on the isotropy subgroup $K.$ Indeed there is a fundamental reason for this fact which will be explained
in the following sections.

\section{Transforms}

In the previous sections irreducible homogeneous vector bundles over $G/K$ with standard fibre determined by the lowest weight $\varpi$
have been generally denoted by $\mathcal O(\varpi)\to G/K.$ In the following sections vector bundles over different flag manifolds will be considered.
In these cases we will sometimes use the abridged notation 
\[
\mathcal O_F(\varpi)
\] 
to denote $\mathcal O(\varpi)\to F.$ We also sometimes write
\[
\mathcal O_K(\varpi)
\]
when the choice of $G$ in $F=G/K$ is celar from context. We do not distinguish a vector bundle from the associated locally free sheaf.

\subsection{Differential geometry of the direct image sheaf construction.}
 We want to apply the direct image sheaf construction (\cite{Grif}, p. 463, \cite{BastonEastwood}, p. 
 50)
to the case in  which we have a bundle of flag manifolds 
\[
\pi: G/K\to G/L,
\] 
with the flag manifold $L/K$ as standard fibre, and we let 
\[
V\to G/K
\] 
be  an irreducible, holomorphic, homogeneous vector bundle 
with non-trivial holomorphic sections. 
In this sense,  $V\to G/K$ is
defined as  $V=G\times_K V_0$ where $V_0$ is an irreducible
$K$-module, and 
\[
H^0(G/K;V)
\] 
is a non-trivial irreducible $G$-module.

By the aforesaid direct image sheaf construction, for all $x\in G/L,$ the restriction $V|\pi^{-1}(x)$ is a holomorphic vector bundle over $\pi^{-1}(x).$ By attaching the vector space of holomorphic sections of this bundle $H^0(\pi^{-1}(x); V|\pi^{-1}(x))$
to each $x\in G/L,$ a new holomorphic vector bundle 
\[
\widetilde V\to G/L
\] 
is obtained, such that 
\[
\widetilde V_x = H^0(\pi^{-1}(x);V|\pi^{-1}(x)).
\] 
Therefore, if 
\[
V=\mathcal O_{K}(\sum_i k_i\varpi_i),
\] 
then 
\[
\widetilde V=\mathcal O_{L}(\sum_i k_i\varpi_i),
\] 
(see \cite{BastonEastwood}). 
The correspondence thus defined by
the direct image sheaf construction will be denoted by $\delta.$

\subsection{Inverse construction.} An inverse procedure can be devised for the aforesaid construction: consider the fibre bundle of flag manifolds 
\[
G/K\to G/L
\] 
and an irreducible homogeneous holomorphic vector
bundle 
\[
\widetilde V\to G/L
\] 
with non-trivial holomorphic sections so that 
\[
\widetilde V= G\times_L \widetilde V_0
\] 
where $\widetilde V_0$ is an irreducible $L$-module. Since $K\subset L$ is a subgroup, $\widetilde V_0$ can
be regarded as a $K$-module which will decompose under the action of $K$  into a sum of irreducible $K$-modules. Let us
denote by $V_0$ the  irreducible $K$-module with the lowest weight,
always appearing in the decomposition of $\widetilde V_0$ with multiplicity one. 
Then 
\[
V=G \times_K V_0
\] 
is the total space of an irreducible homogeneous holomorphic vector bundle over $G/K.$
This inverse procedure actually reverses the direct image sheaf construction above and
ensures that if 
\[
\widetilde V=\mathcal O_{L}(\sum_i k_i\varpi_i)
\] 
then 
\[
V=\mathcal O_{K}(\sum_ik_i\varpi_i).
\] 
We will use $\gamma$
to denote the correspondence defined by the inverse construction.\\

By the Bott-Borel-Weil theorem, we have the `Ward correspondence' 
\[
H^0(G/K;V)=H^0(G/L;\widetilde V),
\] 
so we obtain

\begin{thm}\label{sheaf}
Let 
\[
G/K\to  G/L
\]  
be a fibre bundle of flag manifolds. 
Then, the direct image sheaf and its inverse 
constructions define an identification between the image equivalence class
of  Einstein-Hermitian holomorphic maps 
\[
f:G/K\to Gr_p(W)
\] 
with gauge condition for irreducible homogeneous vector bundles, and  the image equivalence class of Einstein-Hermitian holomorphic maps 
\[
\widetilde f:G/L\to Gr_{\widetilde p}(W)\] 
with gauge condition for irreducible homogeneous vector bundles.
\end{thm}

\subsection{Functoriality}
A categorical description of the
construction of full holomorphic maps, satisfying the gauge condition for semi-positive homogeneous vector bundles, of flag manifolds into complex Grassmannians
has been given in \cite{NaCat} (see also \S 3.4).
In this category, a morphism is a $G$-equivariant bundle isomorphism. 
In the light of this description, the direct image sheaf and the inverse construction provide functors between
the appropriate categories of maps that will be denoted by the same symbols $\delta$ and $\gamma.$

Let 
\[
\phi:V_1\to V_2
\] 
denote a $G$-equivariant bundle isomorphisms on $V_{i}\to G/K$
$(i=1,2).$ Then, $\phi$ induces an intertwining operator 
\[
\Phi:H^0(G/K;V_1)\to H^0(G/K;V_2)
\] 
which is an isomorphism. Suppose that we obtain 
\[
\tilde{V}_i\to G/L
\] 
by the direct image sheaf construction. Since 
\[
H^0(G/K;{V}_{i})=H^0(G/L;\tilde{V}_{i})
\] 
by the construction, the evaluation map 
\[
ev_{i}:H^0(G/L;\tilde{V}_{i})\to \tilde{V}_{i}
\] 
and its adjoint bundle map provide with a $G$-equivariant
bundle  isomorphism 
\[
ev_2 \Phi ev_1^{\ast}:\tilde{V}_1\to \tilde{V}_2.
\]

Conversely, let 
\[
\tilde{\phi}:\tilde{V}_1\to \tilde{V}_2
\] 
be a $G$-equivariant
bundle isomorphism on $\tilde{V}_{i}\to G/L$ $(i=1,2).$ In a similar way, we
obtain a $G$-equivariant bundle isomorphism using the evaluation maps and the induced intertwining operator on the cohomologies.

These observations about functors, together with Theorem \ref{sheaf}, directly imply 

\begin{thm}
  Let $G$ be a compact, connected, simply connected, semi-simple Lie group, and let $T$ be a standard subtorus in $G.$
  Denote by $\mathcal H_T$ the category  of full maps 
\[
f:G/C(T)\to Gr_p(\mathbf{C}^n)
\] 
satisfying the gauge condition for 
some semi-positive homogeneous vector bundle $\oplus_i(V_i,h_i)\to G/C(T)$.  
%$W=H^0(G/C(T);\oplus_i(V_i,h_i)).$

  If $K=C(T_K)$ and $L=C(T_L)$ denote centralisers of appropriate standard tori $T_K,T_L\subset G,$
  % in subsections \S 4.1-4.2,
  then the direct image sheaf construction $\delta$ and the inverse construction $\gamma$ give a bijective
  functor 
\[
\beta: {\mathcal H_{T_K}}\to{\mathcal H_{T_L}},
\] 
such that the following diagram is commutative
\[
  \begin{tikzcd}
    &   {\mathcal H_{T_{K\cap L}}}  \arrow{rdd}{\delta} &   \\
    &                                                                        &   \\
    {\mathcal H_{T_K}}  \arrow{rr}[swap]{\beta} 
                                          \arrow{ruu}[]{\gamma}  &                   & {\mathcal H_{T_L}} 
  \end{tikzcd}
\]
\end{thm}

\subsection{A first example}
Let us give a brief example. Consider the $\mathrm P^1(\mathbf C)$-fibre bundle of flag manifolds
\[
\pi: G/K\to G/L
\] 
where 
\[
G/K=\frac{SU(3)}{U(1)\times U(1)}= F_{1,2}(\mathbf C^3)= F_{SU(3)}(\alpha_1,\alpha_2) = 
\dynkin[
labels={,},scale=1.8]A{XX}
,\]  
\[
G/L=\frac{SU(3)}{U(2)}=\mathrm P^2(\mathbf C)= F_{SU(3)}(\alpha_1)=
\dynkin[
labels={,},scale=1.8]A{X*},
\]
each flag $F$ equipped with an irreducible homogeneous vector bundle as follows: for $G/K$ consider
\[
\mathcal O(2\varpi_2)\to G/K,\quad 
 \text{or}\quad  \dynkin[labels*={,2},scale=1.8]A{XX}.
\] 
This is the homogeneous line bundle with total space 
\[
G\times_K \mathbf C_{-2\varpi_2}
\]
the standard fibre being the one-dimensional $K$-module $\mathbf C_{-2\varpi_2}.$ By Bott-Borel-Weil theorem 
\[
W=H^0(G/K;\mathcal O_K(2\varpi_2))
\] 
is the $G$-module of lowest weight $2\varpi_2,$ that is, 
\[
\dynkin[
labels*={,2},scale=1.8]A{**},
\] 
and the construction above leads to
an Einstein-Hermitian holomorphic map into 
\[
Gr_{\dim W-1}(W)=\mathrm P(W^*).
\]
The direct image sheaf construction produces an irreducible homogeneous vector bundle over $G/L,$ possibly of higher rank,
which corresponds diagrammatically to 
\[\dynkin[%labels*={,},
labels*={,2},scale=1.8]A{X*},
\] 
and which will be denoted by 
\[
\mathcal O(2\varpi_2)\to G/L.
\]
By the direct image sheaf construction, the standard fibre of $\mathcal O(2\varpi_2)\to G/L$ is the space of holomorphic sections
of $\mathcal O_K(2\varpi_2)|_{L/K}$ where $L/K$ stands for the $\mathrm P^1(\mathbf C)$ fibre. This bundle is not other than $\mathcal O(2)\to \mathrm P^1,$ with standard fibre the $U(1)$-module $\mathbf C_{-2}$ (where the subindex $-2$ refers to the weight of the representation), and 
 the space of holomorphic sections given
by the $SU(2)$-module 
\[
S^2\mathbf C^2,\quad \text{or} \quad \dynkin[
labels*={2},scale=1.8]A{*}.
\]
The space of holomorphic sections $W$ of $\mathcal O(2\varpi_2)\to G/L$ is again 
\[
\dynkin[labels*={,2},
labels={,},scale=1.8]A{**},
\] 
but now the evaluation homomorphism at a given point
is a mapping 
\[
W\to S^2\mathbf C^2.
\]
In this particular case, these dimensions are easily obtained by e.g., Weyl's 
dimension formula,
and a map 
\[
\mathrm P^2(\mathbf C)\to Gr_3(\mathbf C^6)
\] 
is obtained.
The previous discussion may be summarised in the following diagram:
\[\begin{CD}
\dynkin[labels*={,2},scale=1.8]A{XX}
 @>>> &
 \mathrm P(\dynkin[labels*={,2},scale=1.8]A{**})\\
 % \mathbf P(\mathbf C^{n+1*})\\ 
 @V\pi VV &  \\
  \dynkin[labels*={,2},scale=1.8]A{X*}  @>>> &  %\mathbf P(\mathbf C^{n+1*})
 Gr_{p}(\dynkin[labels*={,2},scale=1.8]A{**})\\
\end{CD}\qquad
\begin{CD}
F_{1,2}(\mathbf C^3) @>>> \mathrm P(\mathbf C^{6*})\\
@V\pi VV & \\
\mathrm P^2(\mathbf C) @>>> Gr_3(\mathbf C^{6*})
\end{CD}\]

or, following the categorical description,  

\[
  \begin{tikzcd}[row sep=0.5em,column sep=1em,cramped,cells={nodes={anchor=center}}]
    &   & \dynkin[labels={,2},scale=1.8]{A}{**} \arrow[ddddrr, "\alpha_{U(1)\times U(1)}" , shift left=1.5ex] \arrow[ddddll, "\alpha_{U(1)\times 1}" ', shift right=1.5ex]  & & \\
    & & & &  \\
 &   &                                         & & \\
 &   &                                            &  & \\
    \dynkin[labels={,2},scale=1.8]{A}{X*}  \arrow[rrrr, "\gamma" ', shift left=0.5ex] &  & &            & \dynkin[labels={,2}, scale=1.8]{A}{XX} 
\end{tikzcd}
\]
that is
\[\mathrm{P}(\mathbf C^3)\to Gr_3(\mathbf C^{6*}) \stackrel{\gamma}{\Longrightarrow} F_{1,2} (\mathbf C^3)\to \mathrm P(\mathbf C^{6*})\]
where  
\[
\dynkin[labels*={,2},scale=1.8]A{**}\in {\mathcal R_{SU(3)}},\quad\text{and} \quad 
\dynkin[labels*={,2},scale=1.8]{A}{XX} \in {\mathcal H_{U(1)\times U(1)}},\,\,
\dynkin[labels*={,2}, scale=1.8]{A}{x*}\in {\mathcal H_{U(1)\times 1}}.
\]
Interestingly enough, choosing $T_K,T_L$ to be $U(1)\times U(1)$
and $1\times U(1)$ respectively leads to the functorial diagram

\[
  \begin{tikzcd}[row sep=0.5em,column sep=1em,cramped,cells={nodes={anchor=center}}]
    &   & \dynkin[labels={,2},scale=1.8]{A}{**} \arrow[ddddrr, "\alpha_{1\times U(1)}" , shift left=1.5ex] \arrow[ddddll, "\alpha_{U(1)\times U(1)}" ', shift right=1.5ex]  & & \\
    & & & &  \\
 &   &                                         & & \\
 &   &                                            &  & \\
 \dynkin[labels={,2},scale=1.8]{A}{XX}  \arrow[rrrr, "\delta", shift left=0.5ex] &  & &            & \dynkin[labels={,2}, scale=1.8]{A}{*X} 
\end{tikzcd}
\]
in which the direct image sheaf functor $\delta$ 
is nothing but the Penrose transform 
(see \cite{BastonEastwood})
\[F_{1,2}(\mathbf C^3)\to \mathrm{P}(\mathbf C^{6*}) \stackrel{\delta}{\Longrightarrow} \mathrm{P}(\mathbf C^{3*}) \to \mathrm{P}(\mathbf C^{6*}) \]

Combinning both diagrams we obtain

\[
  \begin{tikzcd}[row sep=0.5em,column sep=1em,cramped,cells={nodes={anchor=center}}]
    &   &  \dynkin[labels={,2},scale=1.8]{A}{XX} \arrow[ddddrr, "\delta" , shift left=1.5ex]   & & \\
    & & & &  \\
 &   &                                         & & \\
 &   &                                            &  & \\
 \dynkin[labels={,2},scale=1.8]{A}{X*} \arrow[uuuurr, "\gamma", shift left=1.5ex] \arrow[rrrr, "\beta" ', shift left=0.5ex] &  & &            & \dynkin[labels={,2}, scale=1.8]{A}{*X} 
\end{tikzcd}
\]
\[\mathrm{P}(\mathbf C^3)\to Gr_3(\mathbf C^{6*})\stackrel{\beta}{\Longrightarrow} \mathrm{P}(\mathbf C^{3*})\to \mathrm{P}(\mathbf C^{6*})\]
\newline
From Theorems \ref{equality_holds} and \ref{subcat} we obtain the following
\begin{thm}\label{preserved_minimality}
  In the subcategory of $\mathcal H_T$ consisting of EH holomorphic maps
this transform preserves the minimality of the $L^2$ norm associated to the mean curvature operators.
\end{thm}

Notice that the subcategory in Theorem \ref{preserved_minimality} contains $\widetilde{\mathcal{H}}_T$ as a proper subcategory, e.g., consider the polystable case.
%%%%%%%%%%%%%%%%%%%%%%%%%%%%%%%%%
%%%%%%%%%%%%%%%%%%%%%%%%%%%%%%%%%%

\section{Transforms on projective spaces}

In the direct image sheaf construction, and in the inverse construction, we can restrict the target space to be $\mathbf C\mathrm P^n.$ 
With this
specialization we can profit from Theorem \ref{sheaf} 
to obtain some new results. 
Indeed, we will show that 
Einstein-Hermitian holomorphic maps 
of certain flag manifolds into projective space
can be obtained from 
holomorphic {\it isometric} embeddings of a different domain flag manifold if the fundamental weights defining the flags are related by an inclusion.

A remark on notation is in order. In \S 3, we used $\{\alpha_1,\alpha_2,\dots,\alpha_r\}$ to denote the simple roots
 of the Lie algebra $\mathfrak g^c.$ 
A flag manifold 
was then defined 
by specifying 
a certain subset of the set  of simple roots of $\mathfrak g^c.$ In the present and
successive sections,
though, 
\[
\{\alpha_1,\alpha_2,\dots,\alpha_r\}
\] 
will represent only the special subset of simple
roots defining the flag manifold. Due to this notational decission $r$ coincides with
the second Betti number of the flag manifold.

\begin{defn} 
  Let 
  $f$ be be a holomorphic map from a flag manifold $F$ into $\mathrm P(\mathbf C^{n+1 *})$
  or into $Gr_n(\mathbf R^{n+2}).$  
If the pull back bundle of the universal quotient bundle by $f$ is isomorphic to 
a line bundle 
\[
\mathcal O(k_1,k_2, \cdots, k_r)\to F,
\] 
then $f$ is said to have degree $(k_1,k_2,\cdots,k_{r})$.   
\end{defn}

Let $f$ be an Einstein-Hermitian holomorphic map of degree $(k_1,k_2,$ $\cdots,k_r).$
Then $f$ is an embedding if and only if 
\[
k_1,k_2,\dots,k_r
\] 
are all positive.
From now on we take the following convention: 
 when $f$ is an embedding, 
we always suppose that the K\"ahler form $\omega$ on $M$ takes the form
\[\omega=-\frac{\sqrt{-1}}{2\pi} \; R^{f^*\mathcal O(1)}\]
as in Theorem \ref{fromagag} so that $f$ is a holomorphic isometric embedding.

From Theorem \ref{sheaf} we obtain

\begin{thm}\label{52}
Let 
\[
f:F\to Gr_n(\mathbf C^{n+1})=\mathrm P({\mathbf C^{n+1}}^{\ast})
\] 
be an  Einstein-Hermitian holomorphic map of non-zero degree 
from a flag manifold $F$ into the projective space. 

 Then there exist a fibre bundle   
\[
\pi:F \to F_1
\] 
with a flag manifold $F_1$ as a base, and 
a holomorphic isometric embedding 
\[
f_1:F_1 \to \mathrm P({\mathbf C^{n+1}}^{\ast})
\]
such that 
\[
f=\pi\circ f_1.
\] 
Such a pair $(F_1,f_1)$ is uniquely determined by $f,$ modulo image equivalence. 
\end{thm}

\begin{proof}
  Without loss of generality, we can assume that the map is full.
If $f$ has degree 
\[
(k_1,k_2,\cdots,k_{r})\quad  \text{with} \quad k_i >0,\quad \text{for} \quad i=1,\cdots,r
\] 
then $f$ is a Kodaira embedding and 
the EH condition yields that $f$ is an isometric embedding by Thorem \ref{fromagag}. 
Thus we may take 
\[
F_1=F\quad  \text{and} \quad f_1=f.
\] 

Suppose that 
the degree of $f$ is 
\[
(k_1,k_2,\cdots,k_{r})\quad \text{with} \quad k_1=0, \quad 
\text{and} \quad k_i>0 \quad \text{for} \quad i=2,\cdots,r.
\] 
The flag manifold $F_1$ is defined by 
$G/K_{1}$, 
where $K_{1}$ is the centraliser of a torus in $\mathfrak{g}$ 
induced by 
\[
\alpha_2, \cdots \alpha_{r}.
\] 
(We remove the cross node on $\alpha_1$ in the Dynkin diagram 
corresponding to $F$.) 
Hence we obtain a fiber bundle 
\[
\pi:F \to F_1.
\] 
 
Next, the line bundle 
\[
\mathcal O(k_2, \cdots, k_r)\to F_1
\] 
is pulled back 
by $\pi$ and we get 
\[
\pi^{\ast}\mathcal O_{F_1}(k_2, \cdots, k_r)=\mathcal O_F(0,k_2, \cdots, k_r).
\] 
By the  Bott-Borel-Weil theorem, 
\[
W=H^0(F;\mathcal O(0,k_2, \cdots, k_r))=H^0(F_1;\mathcal O(k_2, \cdots, k_r)). 
\]
The evaluation maps are denoted by 
\[
ev:H^0(F;\mathcal O(0,k_2, \cdots, k_r)) \to \mathcal O(0,k_2, \cdots, k_r)
\] 
and 
\[
ev_1:H^0(F_1;\mathcal O(k_2, \cdots, k_r)) \to \mathcal O(k_2, \cdots, k_r),
\] 
respectively.
Since we have 
\[
 ev_x(\pi^*t)= ev_{\pi(x)}(t)
\] 
for $x\in F$ and $t\in H^0( F_1;\mathcal O(k_2,\cdots,k_r))$  
the 
map 
\[
f_0:F \to \mathrm P({W}^{\ast}),
\]
induced by 
\[
(\mathcal O_{F}(0,k_2,\dots, k_r),\; W)
\]
factors through $\pi$, 
in other words, 
\[
f_0=f_{01}\circ \pi,
\] 
where \[
f_{01}:F_1 \to \mathrm P({W}^{\ast})
\] 
is the map
induced by 
\[
(\mathcal O_{F_1}(k_2,k_3,\dots,k_r),W).
\]
Since $k_i>0$ for all $i=2,3,\dots,r$ the map is a holomorphic isometric
embedding by Theorem \ref{fromagag}.
 From Theorem \ref{subcat}, 
we conclude 
\[
f=f_{0}
\] 
and may take 
\[
F_1=G/K_{1}, \quad  \text{and} \quad f_1=f_{01}.
\] 
\end{proof}

The converse of the previous theorem is immediate. 

\begin{thm}\label{converse}
Let 
\[
f:F_G(\alpha_1,\alpha_2, \cdots \alpha_{r})
\to Gr_n(\mathbf C^{n+1})=\mathrm P({\mathbf C^{n+1}}^{\ast})
\] 
be a holomorphic isometric embedding. 

Then for any 
set of simple roots
$\{\alpha_1,\alpha_2, \cdots \alpha_{s}\}$ satisfying 
\[\{\alpha_1,\alpha_2, \cdots \alpha_{r}\}\subset 
\{\alpha_1,\alpha_2, \cdots \alpha_{s}\}, \]
there exists a fibre bundle   
\[\pi:F_G(\alpha_1,\alpha_2, \cdots \alpha_{s}) 
  \to F_G(\alpha_1,\alpha_2, \cdots \alpha_{r})\] 
such that  
\[
\tilde f=\pi\circ f:F_G(\alpha_1,\alpha_2, \cdots \alpha_{s}) 
\to \mathrm P({\mathbf C^{n+1}}^{\ast})
\]
is an Einstein-Hermitian holomorphic map. 
\end{thm}

\begin{proof}
  The line bundle 
\[
\mathcal O_{F_1}(k_1,k_2,\dots,k_r) \to F_1(\alpha_1,\alpha_2,\dots,\alpha_r), \quad k_i>0
\] 
is pulled back to 
\[
\mathcal O_{ F}(k_1,\dots,0,\dots,k_s)= \pi^*\mathcal O_{ F_1}(k_1,k_2,\dots,k_s)
\to F_G(\alpha_1,\alpha_2, \cdots \alpha_{s}), 
\]
with $r<s$. 
The inverse construction
  yields the result.
  \end{proof}

In order to clarify the  theorem, consider the following special
case: let $V\to G/K$ be the  line bundle 
\[
\mathcal O(2\varpi_2)\to F_{SU(4)}(\alpha_1,\alpha_2,\alpha_3).
\] 
This is the 
homogeneous line bundle with total space defined by 
\[
\mathcal O(2\varpi_2)=G\times_K \mathbf C_{-2\varpi_2}.
\]  
This corresponds to the diagram
\[
\dynkin[labels*={,2,},
labels={,,},scale=1.8]A{XXX}.
\]

The projection of flag manifolds 
\[\pi: G/K\to G/L
\] 
described by
uncrossing the first node is 
\[
F_{SU(4)}(\alpha_1,\alpha_2,\alpha_3)\to F_{SU(4)}(\alpha_2,\alpha_3).
\] 
The $K$-module $\mathbf C_{-2\varpi_2}$ can be regarded as an $L$-module and 
the direct
image sheaf construction gives a new  homogeneous line bundle 
\[
\mathcal O(2 \varpi_2)\to F_{SU(4)}(\alpha_2,\alpha_3)
\]
with total space   
\[
G\times_L\mathbf C_{-2\varpi_2}.
\]
Still, we could take the construction one step further and, by
uncrossing the node corresponding to $\alpha_3$ get a new
projection 
\[F_{SU(4)}(\alpha_2,\alpha_3)\to Gr_2(\mathbf C^4)
\] 
together with the irredicible homogeneous line bundle 
\[
\mathcal O(2\varpi_2)\to Gr_2(\mathbf C^4).
\]
This tower of flags, each equipped with the appropriate  homogeneous line bundle can be described by the diagram
\[\begin{CD}
\dynkin[labels*={,2,},
labels={,,},scale=1.8]A{XXX}
 @>\pi_1>>  
 \dynkin[labels*={,2,},
labels={,,},scale=1.8]A{*XX} 
@>\pi_2>>
\dynkin[labels*={,2,},
labels={,,},scale=1.8]A{*X*} 
\end{CD}\]
By general considerations on the direct image sheaf construction the space of holomorphic sections of the aforesaid bundles are isomorphic, and by the Bott-Borel-Weil theorem 
\[
W=H^0(F,\mathcal O(2\varpi_2)),
\]
where $F$ stands for any of the flags in the tower, is isomorphic to the irreducible $G$-module of highest weight $2\varpi_2,$ that is,
\[
\dynkin[labels*={,2,},
labels={,,},scale=1.8]A{***}.
\] 
In each case, the induced map by 
\[
(\mathcal O_F(2\varpi_2)),W)
\]
gives an Einstein-Hermitian holomorphic map  
\[
F\to Gr_{\dim W - 1}(W) = \mathrm{P}(W^*).
\]
This correspondence is expressed in diagrammatic form as follows:
\[\begin{CD}
\dynkin[labels*={,2,},
labels={,,},scale=1.8]A{XXX}
 @>\tilde f>> &
 \mathrm P(\dynkin[labels*={,2,},
labels={,,},scale=1.8]A{***})\\
 @V\pi_1 VV & @| \\
  \dynkin[labels*={,2,},
labels={,,},scale=1.8]A{*XX} @>>> &  
 \mathrm P(\dynkin[labels*={,2,},
labels={,,},scale=1.8]A{***})\\
@V\pi_2 VV & @| \\
\dynkin[labels*={,2,},
labels={,,},scale=1.8]A{*X*} @>> f > &  
 \mathrm P(\dynkin[labels*={,2,},
labels={,,},scale=1.8]A{***})\\
\end{CD} 
\qquad
\qquad
\begin{CD}
F_{1,2,3}(\mathbf C^4) @>\tilde f>> &
 \mathrm P(\mathbf C^{20*})\\
 @V\pi_1 VV & @| \\
F_{1,2}(\mathbf C^4)@>>> &  
\mathrm P(\mathbf C^{20*})\\
@V\pi_2 VV & @| \\
Gr_2(\mathbf C^4)@>> f > &  
 \mathrm P(\mathbf C^{20*})\\
\end{CD}
\]
or also,
\[
  \begin{tikzcd}[row sep=0.5em,column sep=1em,cramped,cells={nodes={anchor=center}}]
    &   & \dynkin[labels={,2,},scale=1.8]{A}{***} \arrow[dddd, "\alpha_{T'} " , shift left=0ex]   \arrow[ddddrr, "\alpha_{T''} " , shift left=1.5ex] \arrow[ddddll, "\alpha_T " ', shift right=1.5ex]  & & \\
    & & & &  \\
 &   &                                         & & \\
 &   &                                            &  & \\
  \dynkin[labels={,2,},scale=1.8]{A}{xxx}  \arrow[rr, "\beta", shift left=0.5ex] &  & \dynkin[labels={,2,},scale=1.8]{A}{*xx} \arrow[ll, "\beta^{-1}", shift left=0.5ex]   \arrow[rr, "\beta", shift left=0.5ex] &            & \dynkin[labels={,2,}, scale=1.8]{A}{*x*} \arrow[ll, "\beta^{-1}", shift left=0.5ex] 
\end{tikzcd}
\]
 where $\alpha_S$ are functors ${\mathcal R_{SU(4)}}\to {\mathcal H_S}$ for $S=\{T,T',T''\}$ with  
\begin{align*}
&T=U(1)^3,\\ 
&T'=S\left(U(2)\times U(1)\times U(1)\right), \\ 
&T''=S\left(U(2)\times U(2)\right).
\end{align*}

\section{Moduli spaces of Einstein-Hermitian holomorphic maps into quadrics}
We obtain a non-discrete moduli space when the target is a quadric, but we still obtain the tower of transforms even in this case.
To do this we need a generalization of do Carmo-Wallach theory.
Detailed accounts of the generalisation of the theory of do Carmo and Wallach can be found in \cite{Na13, Na15}, where the interested reader will find complete proofs of the central results.  Abridged versions, suitable to grasp the essentials, can be found in \cite{MNT, MacNag}, \S 2, and \cite{MN22, MNxx} \S 2.\\
Based on the theory developed in the aforesaid articles, the following proposition characterises full holomorphic maps from a  compact K\"ahler manifold $M$ into
quadrics as the maps induced by a Hermitian line bundle  $L \to M,$ and its space of holomorphic sections.
By restricting the target space to quadrics, Theorem 2.1 of \cite{MNxx} can be simpified, and now requires a single equation when two equations were previously
necessary (cf \cite{MNxx}). This new version of the theorem simplifies our argument drastically, avoiding the heavy representation theoretic computations previously needed.
%%%%%%%%%%

When we consider a quadric, it can be identified with 
 the Grassmannian $Gr_n(\mathbf R^{n+2})$ of 
oriented $n$-planes in an oriented $\mathbf R^{n+2}$. 
Then the universal quotient bundle 
\[
Q \to  Gr_n(\mathbf R^{n+2})
\]
has an orientation and a fibre metric $g_Q$.   
Since $Q \to  Gr_n(\mathbf R^{n+2})$ is of real rank two, 
it is recognised as a Hermitian holomorphic line bundle 
with a Hermitian metric $h_Q$ such that $g_Q=\text{Re}\,h_Q$.

A mapping into a quadric $f : M \to Gr_n (\mathbf R^{n+2} )$ is said to be {\it full} if the induced linear
map 
\[
\mathbf R^{n+2} \to \Gamma (f^*  Q)
\] 
is a monomorphism. To state the theorem, we introduce a new equivalence relation on induced maps. 

\begin{defn}
Let $f_1$ and $f_2:M\to Gr_n(\mathbf R^{n+2})$
be smooth maps.
Then $f_1$ is said to be {\it gauge equivalent} to $f_2$,
if there exists an isometry 
$\psi\in O(n+2)$
of $Gr_n(\mathbf R^{n+2})$
such that $\psi$ provides with a bundle isomorphism
\[
f_{1}^{\ast}Q \to f_{2}^{\ast}Q.
\]
This means that we have a linear isomorphism
\[
\psi:Q_{f_1(x)} \to Q_{f_2(x)}
\]
for any $x \in M$,  
where $Q_{f_i(x)}$ are considered as subspaces of $\mathbf R^{n+2}$.  
\end{defn}

\begin{rem}
If $f_1$ and $f_2$ are gauge equivalent, then we have 
\[
f_2(x)=\psi f_1(x),
\]
where $f_{i}(x)$ are regarded as subspaces of $\mathbf R^{n+2}$ for any $x \in M$.
This relation is called the {\it image equivalence} in the original do Carmo-Wallach theory
\cite{DoC-Wal}.
\end{rem}

\begin{thm}\label{HGenDWI}
Let $(L,h_L)$ be a Hermitian holomorphic line bundle over 
a compact K\"ahler manifold $M.$ 
We regard $L$ as an oriented real vector bundle 
with a metric $g_{L}=\text{Re}\,h_L.$  
Let $W$ denote  the space of holomorphic sections of $L$ 
equipped with the $L^2$ Hermitian inner product. 
We regard $W$ as a real vector space with the $L^2$ inner product $(\cdot,\cdot)_W$ and the orientation 
induced by the complex structure. 

Let 
\[
f:M \to Gr_n(\mathbf R^{n+2})
\] 
be a full holomorphic map into 
a quadric with the metric of Fubini-Study type 
satisfying the gauge condition{\rm :} 

\noindent{\rm (i)} 
The pull-back bundle 
$(f^{\ast}Q, f^{\ast}g_Q) \to M$  
is isomorphic to $(L,g_L)$ as an oriented vector bundle with fibre metric. 

Then we have a positive semi-definite symmetric endomorphism 
\[
T\in \text{\rm End}\,(W)
\]
satisfying the following three conditions{\rm :}

\noindent {\rm (I)} The vector space $\mathbf R^{n+2}$ is a subspace of $W$ with the inclusion 
$\iota:\mathbf R^{n+2} \to W$ which preserves inner products 
and 
$L\to M$ is globally generated by $\mathbf R^{n+2}$.

\noindent {\rm (II)} 
As a subspace, $\mathbf R^{n+2}=\text{\rm Ker}\,T^{\bot}$ and 
the restriction of $T$ to $\mathbf R^{n+2}$ is a positive symmetric endomorphism. 

\noindent {\rm (III)} 
The endomorphism $T$ satisfies 
%%%%%%%%%%%%%%%%%%%%%%%%%%% DW 2 %%%%%%%%%%%%%%%%%%%%%%%%%%%%%%%%%
\begin{equation}\label{HDW 4}
ev \circ T^2 \circ ev^{\ast}=Id_L, 
\end{equation}
where $ev:\underline{W}\to L$ is the evaluation map and $ev^{\ast}$ is its adjoint map.

If $\iota^{\ast}:W \to \mathbf R^{n+2}$ denotes the adjoint linear map of 
$\iota:\mathbf R^{n+2} \to W$, 
then $f:M \to Gr_n(\mathbf R^{n+2})$ is realized as the induced map 
by $\left(V,\mathbf R^{n+2},\iota\left(\iota^{\ast} T \iota\right)\right)${\rm :} 
%%%%%%%%%%%%%%%%%%%%%%%%%%%% DW 3 %%%%%%%%%%%%%%%%%%%%%%%%%%%%%%%%
\begin{equation}\label{HDW5} 
f\left(x \right)=\text{\rm Ker}\,(ev\circ \iota\circ \left(\iota^{\ast} T \iota\right))_{x}, 
\quad x \in M,
\end{equation}  
where the orientation of $\text{\rm Ker}\,(ev\circ \iota\circ \left(\iota^{\ast} T \iota\right))_{x}$ is given by those of $L_x$ and $\mathbf R^{n+2}$. 
Moreover, if the orientation of $\text{Ker}\,T$ is fixed, then 
we have a unique holomorphic totally geodesic embedding of 
$Gr_n(\mathbf R^{n+2})$ into $Gr_{n^{\prime}}(W)$ 
by $\iota\left(\mathbf R^{n+2}\right)=\text{\rm Ker}\,T^{\bot}$, 
where $n^{\prime}=n+\text{\rm dim}\,\text{\rm Ker}\,T$
and a bundle isomorphism   
$\left(ev \circ \iota\circ\left(\iota^{\ast} T \iota\right)\right)^{\ast}
:L\to f^{\ast}Q.$ 

Conversely, 
suppose that a vector space $\mathbf R^{n+2}$ with an inner product and an orientation, and 
a positive semi-definite symmetric endomorphism 
\[
T\in \text{\rm End}\,(W)
\] 
satisfying 
conditions {\rm (I)}, {\rm (II)} 
and {\rm (III)} are given.  
Then we have a unique holomorphic totally geodesic embedding of 
$Gr_n(\mathbf R^{n+2})$ into $Gr_{n^{\prime}}(W)$ after fixing the orientation of 
$\text{Ker}\,T$  
and 
the map 
\[
f:M \to Gr_{p}(\mathbf R^{n+2})
\] 
defined by \eqref{HDW5}
is a full holomorphic map into $Gr_n(\mathbf R^{n+2})$ 
satisfying the gauge condition {\rm (i)} with bundle isomorphism $L\cong f^{\ast}Q.$

Suppose that $f_i\;(i=1,2)$ is the map induced by the triple 
\[
(V,\mathbf R^{n+2},\iota (\iota^{\ast} T\iota)),
\]
with inclusion $\iota,$ such that $\iota (\mathbf R^{n+2}) =\ker T_i^\perp.$ Then $f_1$ is gauge equivalent
to $f_2$ if and only if $T_1=T_2.$

\end{thm}
%%%%%%%%%%%%%%%%%%%%%%%%%%%%%%%%%%%%%%%%%%%%%%%%%%%%%%%%%%%%%%%%%%%%%%
\begin{proof}
The proof is almost the same as the proof of Theorem 5.9 in \cite{Na15}. 
The only difference is that we need only one equation in the condition (III). 

Identifying $f^{\ast}Q \to M$ and $L\to M$, we see from \eqref{HDW 4} that 
$g_L=\text{\rm Re}\,h_L$ coincides with 
$g_{Q}=\text{\rm Re}\,h_{f^{\ast}Q}$. 
Since these are 
compatible with the complex structures, 
$h_{f^{\ast}Q}$ also coincides with $h_L$. 
Thus, $(f^{\ast}Q, f^{\ast}h_Q)$ is isomorphic to 
$(L,h_L)$ as a Hermitian holomorphic line bundle.  
Then the uniqueness of the Hermitian connection yields that 
$f$ satisfies the gauge condition for $(L,h_{L})$.  
\end{proof}

We can specialize Theorem \ref{HGenDWI} 
to the case where the domain manifold is 
a flag manifold $G/K$ in the way in the proof of 
Theorem 5.24 in \cite{Na13}. 
To do this, 
notice that every holomorphic line bundle on $G/K$ is a homogeneous vector bundle. 
Let 
\[
L\to G/K
\] 
be a 
holomorphic line bundle which is supposed to be an associated bundle :\ $L=G\times_K V_0$,  
where $V_0$ is a unitary $1$-dimensional 
$K$-module.  
By Schur's lemma, an invariant Hermitian metric $h$ on $L$ is unique up to a positive constant multiple. 

If $L\to G/K$ has a non-trivial holomorphic section, 
then the Bott-Borel-Weil theorem implies that the space of holomorphic 
sections of $L\to G/K$ denoted by $W$ is an irreducible unitary $G$-module 
and globally generates the bundle. 
We equip $W$ with a $G$-invariant $L^2$ Hermitian inner product. 
 We restrict the coefficient field of the complex vector bundle to $\mathbf R$ to obtain 
an {\it oriented real} vector bundle denoted by the same symbol $L\to G/K$ with 
a fibre-metric induced by the real part of $h$, where  
the orientation is given by the complex structure. 
In a similar way, 
we regard $W$ as an 
{\it oriented real} 
vector space with $G$-invariant $L^2$ inner product 
denoted by $(\cdot,\cdot)_W$.  
Then, we obtain the induced map into an oriented real Grassmannian 
$Gr_p(W)$, 
which is said to be {\it standard}, 
 where $p=\text{dim}_{\mathbf R}\,W-2$.
Hence the standard map is the composition of the induced map by $(L,W)$ 
into a projective space $\mathrm P(W^{\ast})$ and 
the totally geodesic embedding of $\mathrm P(W^{\ast})$ into $Gr_p(W)$.

With these understood, 
from Lemma 5.17 in [11], we derive 
%%%%%%%%%%%%%%%%%%%%%%%%%%%%%%%%%%%%%%%%%%%%%%%%%%%%%%%
\begin{lemma}\label{subsp} 
The $K$-module $V_0$ can be regarded as a complex subspace of $W$. 
\end{lemma}

Denote by $U_0$ the orthogonal complement of $V_0$ in $W$ 
 so that $p=\text{dim}\,U_0$.
Then, the standard map denoted by 
\[
f_0: G/K\to Gr_p(W)
\] 
is explicitly written down as 
\[  f_0([g])=gU_0 \subset W,  \quad \text{for all}\,\, [g]\in G/K,\quad g \in G,
\]
 which is a $G$-equivariant map.

Next, 
$\mathrm{S}(W)$ denotes  
the set of symmetric 
endomorphisms of $W$.  
We equip $\mathrm{S}(W)$ 
with a $G$-invariant inner product 
\[
(A,B)_{_S}=\text{trace}\,AB, 
\]
for $A,B \in \mathrm{S}(W)$. 
Define a symmetric transformation $\mathrm{S}(u,v)$ for $u$, $v \in W$ as 
\[ 
\mathrm{S}(u,v):=\frac{1}{2}\left\{u\otimes (\cdot, v)_{_W} + v\otimes (\cdot, u)_{_W} \right\}.
\]
If $U$ and $V$ are subspaces of $W$, 
we define a real subspace  $\mathrm{S}(U,V) \subset \mathrm{S}(W)$  
spanned by $\mathrm{S}(u,v)$ 
 where $u\in U$ and $v\in V$. 
In a similar fashion, $G\mathrm{S}(U,V)$  
denotes the subspace of $\mathrm{S}(W)$  spanned by 
$g\mathrm{S}(u,v),$ where $g\in G.$ 

Then the equation \eqref{HDW 4} is equivalent to 
\begin{equation}\label{HDW6}
\left(T^2- I\!d_W, G\mathrm{S}(V_0, V_0)\right)_{S}=0, 
\end{equation}
and the resulting map is written as:
\[
f\left([g]\right)=\left(\iota^{\ast}T\iota \right)^{-1}\left(f_0\left([g]\right)\cap \text{\rm Ker}\,T^{\bot}\right). 
\] 
\begin{rem}
  The previuos expression will be abbreviated to 
\[
f([g])=T^{-1}f_0([g])
\] 
  under the convention that the inverse of $T$ is taken on 
  $\mathrm{Ker}\; T^\perp.$
  \end{rem}
Let 
\[
f:M\to Gr_n(\mathbf R^{n+2})
\] 
be a full EH-holomorphic  
map of a K\"ahler manifold $M$ into a quadric. 
We denote the pull-back line bundle by 
\[
L\to M,
\] 
instead of $f^{\ast}Q \to M$. 
The space of holomorphic sections of $L\to M$ is denoted by $H^0(M;L)$  
with the evaluation map $ev:\underline{H^0(M;L)} \to L$. 
To emphasize the holomorphic structure, 
we denote by $J_L$ the complex structure on $L\to M,$ and denote by $J$ the
complex structure on $H^0(M;L).$   
 From Theorem \ref{HGenDWI}, 
we deduce that $\mathbf R^{n+2}$ is a real subspace of $H^0(M;L)$. 

Suppose that $\mathbf R^{n+2}$ is a {\it complex}\, subspace of $H^0(M;L)$ 
and the inner product on $\mathbf R^{n+2}$ is compatible with the complex structure.  
Since $\mathbf R^{n+2}$ has a $\text{U}(m+1)$-structure $(2m=n)$ from the assumption, 
we have that 
\[
S(\mathbf R^{n+2})=H(\mathbf R^{n+2}) \oplus H(\mathbf R^{n+2})^\perp, 
\]
as $\text{U}(m+1)$-module. In the notation of \cite{Na15} this 
orthogonal complement $H(\mathbf R^{n+2})^\perp$
was denoted  $Sym_I(\mathbf R^{n+2}).$ 
Notice that this space has an invariant complex structure and is therefore a unitary representation 
equivalent to $S^2\mathbf C^{m+1}$. 

We need a version of Theorem 6.10 in \cite{Na15};
we can replace holomorphic isometric immersions 
by Einstein-Hermitian holomorphic maps in it, 
since a generalisation of Calabi's rigidity theorem plays an essential role in the proof 
and it is valid for Einstein-Hermitian holomorphic maps.
%%%%%%%%%%%%%%%%%%%%%%%%%%%%%%%%%%%%%%%%%%%%%%%%%%%%%%%%%%
\begin{thm}\label{quadmodcpx}
Let $(L,h)$ be a Hermitian holomorphic line bundle over a K\"ahler manifold 
$M$. 
Let 
$\mathcal M$ be the moduli space of 
full 
Einstein-Hermitian holomorphic maps 
of 
$M$ into a quadric $Gr_n(\mathbf R^{n+2})$ 
with the gauge condition for $(L,h)$ modulo gauge equivalence relation of maps.   

Suppose that  $\mathbf R^{n+2}$ is a {\it complex}\, subspace of $H^0(M;L)$ and 
the inner product on $\mathbf R^{n+2}$ is compatible with the complex structure. 
If there exists $f \in \mathcal M$ such that the evaluation map 
$ev:\underline{\mathbf R^{n+2}}\to L$ by $f$ satisfies $J_Lev=evJ$,  
then $\mathcal M$ has the induced complex structure and 
is an open submanifold of a complex  
subspace of $H(\mathbf R^{n+2})^\perp$. 
\end{thm}

%%%%%%%%%%

As described in the general theory, certain representation
spaces will be specially relevant in the description of moduli
spaces of maps. We will use the theory of the lowest weight to identify these representations.

\begin{defn}
Let $G$ be a compact connected semisimple Lie group 
and $W$ a unitary representation space of $G$ with 
\[
-(k_1\varpi_1 +k_2 \varpi_2+\cdots 
+k_r \varpi_{r})
\]
as the {\it lowest weight}. 
We denote by $S^2W$ the symmetric power of $W$ of degree $2$. 
When $S^2W$ is irreducibly decomposed, 
$S^2W$ has a unique irreducible summand $W_1$ with  
\[
-2(k_1\varpi_1 +k_2 \varpi_2+\cdots 
+k_r \varpi_{r})
\]
as the lowest weight. 
Then the orthogonal complement of $W_1$ is denoted 
$S^2W_{0}$:
\[
S^2W=W_1\oplus_{\perp} S^2W_0. 
\] 
\end{defn}
%%%%%%%%%%%%%%%%%%%%%%%%%%%%%%%%%%%%%%%%%%%%%%%%%%

\begin{rem}
The map 
\[
\tau:Gr_n(\mathbf R^{n+2})\to Gr_n(\mathbf R^{n+2})
\]
defined by switching the orientation of the $n$-planes of $\mathbf R^{n+2}$ is
an isometry. From now on we will not make distiction between 
$f:M\to Gr_n(\mathbf R^{n+2})$ and $\tau \circ f$.  
  \end{rem}

Using all this notation, the theorem describing the moduli space
of maps is stated as follows:

\begin{thm}\label{mod}
Let $G/K$ be a flag manifold.
We denote by 
$W$ the space of holomorphic sections of 
\[
\mathcal O(k_1,k_2, \cdots, k_r)\to G/K,
\] 
which is considered as a real vector space $\mathbf R^{n+2}$ 
with the orientation induced by the complex structure and 
the inner product induced from the $L^2$ Hermitian inner product 
up to a constant multiple.  

If 
\[
f:G/K \to  Gr_m(\mathbf R^{m+2})
\]
is a full Einstein-Hermitian holomorphic map with degree $(k_1,k_2,\cdots,k_{r})$, then 
\[
m\leqq n.
\]  

We define a subset  
$\mathcal M$ of $S^2W_0$ with the induced topology by {\rm :}
\[
\mathcal M=\left\{C \in S^2W_{0}\,|\, Id_{W}+C > 0 \right\}, 
\]
where $S^2W_{0}$ is regarded as a subspace of the space of 
symmetric transformations on $W$. 
Then, 
$\mathcal M$ is 
 identified with
the moduli space 
of full Einstein-Hermitian holomorphic maps of $G/K$ into $Gr_n(\mathbf R^{n+2})$ 
with degree $(k_1,k_2,\cdots,k_{r})$ up to gauge equivalence. 

When $\overline{\mathcal M}$ denotes the closure of 
$\mathcal M$,   
every boundary point of $\overline{\mathcal M}$ 
distinguishes a subspace $\mathbf R^{p+2}$ of $\mathbf R^{n+2}$ and 
describes a full Einstein-Hermitian holomorphic map into 
$Gr_p(\mathbf R^{p+2})$ 
which can be regarded as a totally geodesic submanifold of 
$Gr_{n}(\mathbf R^{n+2})$. 

The moduli space $\mathbf M$ 
of these maps up to image equivalence 
is identified with 
the quotient space of $\overline{\mathcal M}$ by $S^1${\rm :} 
\[
\mathbf M= \overline{\mathcal M}/S^1. 
\] 
Each $C\in \overline{\mathcal M}$ represents a map \[(Id_{W}+C)^{-\frac{1}{2}}f_{0}.\]
\end{thm}

\begin{proof}

Since the standard map is the composition of the Kodaira embedding 
\[
G/K \to \mathrm P(W^{\ast})\] 
and the totally geodesic embedding  
\[
\mathrm P(W^{\ast}) \to Gr_{n}(\mathbf R^{n+2}),
\] 
and $W=\mathbf R^{n+2},$ it follows from
Theorem 3.2 that the standard map by 
\[
(\mathcal O(k_1,k_2,\cdots,k_r)\to G/K,W)
\] 
is a full EH holomorphic map. 
Then, Theorem \ref{quadmodcpx} yields that  $\mathcal M$ is an open submanifold
of a complex  subspace of $H(\mathbf R^{n+2})^\perp$. 
  
For any positive symmetric endomorphism $T$ on $\mathbf R^{n+2}$, 
We take a symmetric endomorphism $C$ on $\mathbf R^{n+2}$ such that 
\[
T^2=Id_{W}+C.
\] 
Then the equation \eqref{HDW6} reduces to 
\[
\left(C, G\mathrm{S}(V_0, V_0)\right)_{S}=0. 
\]
Thus we need to identify the $G$-representation subspace 
\[
G\mathrm{S}(V_0, V_0)^\perp \subset S^2(W).
\]

Since we should regard $W$ as an orthogonal $G$-module when applying
Theorem 6.1  
  we must consider an orthogonal irreducible decomposition
of the space of symmetric endomorphisms of $W$ regarded as a $G$-module. However, since
the complex structure of $H(W)^\perp$ is invariant and $G \mathrm{S}(V_0,V_0)^\perp$ is a complex subspace (by Theorem 6.3), 
we may decompose $H(W)^\perp$ 
into irreducible unitary representations.

Let $V_0$ be the one-dimensional complex subspace spanned by 
the lowest vector $v$ in $W$. 
Then $S(V_0,V_0)$ is spanned by 
\[
v \otimes v
\] and so, 
$S(V_0,V_0)$ is a subspace of $W_1$. 
We therefore obtain 
\[
GS(V_0,V_0)=W_1.
\] 

For the moduli space  modulo image equivalence we need to consider the action on $\mathcal M$ 
of the  centraliser of the holonomy group of the connection in the structure group of 
\[
\mathcal{O}(k_1,k_2,\dots,k_r)\to G/K,
\]
regarded as a real vector bundle,  i.e., $S^1.$
For more details, see \cite{MNT}, \cite{Na13} and \cite{Na15}.

Finally, from Theorem 6.1,  $C \in \mathcal M$ corresponds to 
\[
(Id_{W}+C)^{-\frac{1}{2}}f_{0},
\] 
where $f_{0}(x)$ stands for an $n$-dimensional subspace of 
$\mathbf R^{n+2}$ for each $x \in G/K$. 
\end{proof}

By virtue of Theorem \ref{fromagag}, 
we obtain the following refinement of the theorem:

\begin{cor}
  Under the same hypotheses of Theorem \ref{mod}, let $k_i>0$ for $i=1,2,\dots,r$ and suppose that the
  K\"ahler form is in the cohomology class represented by the first Chern class of the pull-back of the
  universal quotient bundle $\mathcal O(1)\to Gr_n(\mathbf R^{n+2}).$ Then the moduli space $\mathbf M$ 
of holomorphic isometric embeddings of $G/K$ into $Gr_n(\mathbf R^{n+2})$ 
with degree $(k_1,k_2,\cdots,k_{r})$ 
is identified with 
the quotient space of $\overline{\mathcal M}$ by $S^1:$ 
\[
\mathbf M=\overline{\mathcal M}/S^1. 
\] 
\end{cor}
\section{Transforms on quadrics}
There are no rigidity results on quadrics, and in fact a moduli space of maps into quadrics is obtained. 
The moduli space does not depend on the isotropy subgroup of the domain flag, so it may still be possible to obtain a correspondence between moduli
spaces. This is the case.

\begin{thm}
Let   
\[
f:F \to Gr_n(\mathbf R^{n+2})
\]
be an Einstein-Hermitian holomorphic map  of 
a flag manifold $F$.     

Then there exist a fibre bundle   
\[
\pi:F \to F_1
\]
with a flag manifold $F_1$ as a base, and 
a holomorphic isometric embedding 
\[
f_1:F \to Gr_n(\mathbf R^{n+2})
\]
such that 
\[
f=\pi\circ f_1.
\] 
Such a pair $(F_1,f_1)$ is uniquely determined by $f$. 
\end{thm}

\begin{proof}
In a similar manner that in the proof of Theorem \ref{52}
we obtain 
 \[
f_0=f_{01}\circ \pi,
\] 
where $f_{0}:F \to Gr_n(\mathbf R^{n+2})$ and  
$f_{01}:F_1 \to Gr_n(\mathbf R^{n+2})$ are the standard maps. 
When $f$ is written as 
\[
f(x)=(Id_{W}+C)^{-\frac{1}{2}}f_{0}(x), \quad x \in F
\]
from Theorem \ref{mod}, we can define 
a holomorphic isometric embedding 
\[
f_1:F_1\to Gr_n(\mathbf R^{n+2})
\] 
as:
\begin{equation}\label{f1_equation}
f_1(p)=(Id_{W}+C)^{-\frac{1}{2}}f_{01}(p), \quad p \in F_1. 
\end{equation}
Then we get $f=\pi \circ f_1$. 
\end{proof}
Taking (\ref{f1_equation}) into account, and following the same steps as in the proof
of Theorem \ref{converse} we get
\begin{thm}
Let  
\[
f:F_G(\alpha_1,\alpha_2, \cdots \alpha_{r}) 
\to Gr_n(\mathbf R^{n+2})
\]
be a holomorphic isometric embedding. 

Then for any 
set of simple roots 
$\{\alpha_1,\alpha_2, \cdots \alpha_{s}\}$ satisfying 
\[
\{\alpha_1,\alpha_2, \cdots, \alpha_{r}\}\subset 
\{\alpha_1,\alpha_2, \cdots, \alpha_{s}\},
\] 
there exists a fibre bundle   
\[
\pi:F_G(\alpha_1,\alpha_2, \cdots, \alpha_{s}) 
\to F_G(\alpha_1,\alpha_2, \cdots, \alpha_{r})
\] 
such that 
\[
\tilde f=\pi\circ f:F_G(\alpha_1,\alpha_2, \cdots \alpha_{s})
\to Gr_n(\mathbf R^{n+2})
\] is an Einstein-Hermitian holomorphic map. 
\end{thm}

From Theorem \ref{equality_holds} we obtain the following analogue of Theorem \ref{preserved_minimality}:
\begin{thm}
  The tower of transforms preserves the minimality of the $L^2$ norm associated to the mean curvature operators.
\end{thm}

\end{document}